\documentclass[a4paper,12pt]{amsart}
\usepackage[final]{epsfig}
\usepackage{graphics}
\usepackage{amsmath}
\usepackage{amsfonts}
\usepackage{latexsym}
\usepackage{amssymb}
\usepackage{fourier}
\usepackage{graphicx}
\usepackage{url}
\usepackage{epstopdf}
\usepackage{hyperref}
\usepackage{xcolor}
\usepackage{comment}
\usepackage{tikz}
\usepackage{marginnote}
\usetikzlibrary{arrows.meta, positioning}

\newtheorem{lemma}{Lemma}[section]
\newtheorem{proposition}[lemma]{Proposition}
\newtheorem{remark}[lemma]{Remark}
\newtheorem{example}[lemma]{Example}
\newtheorem{theorem}{Theorem}

\newtheorem{corollary}[lemma]{Corollary}

\newcommand{\g}{{\gamma}}

\newcommand{\eps}{{\varepsilon}}
\newcommand{\proofend}{$\Box$\bigskip}
\newcommand{\C}{{\mathbb C}}

\newcommand{\R}{{\mathbb R}}
\newcommand{\Z}{{\mathbb Z}}

\def\proof{\paragraph{Proof.}}

\renewcommand{\P}{\mathcal{P}}
\newcommand{\D}{\mathcal{D}}
\renewcommand{\C}{\mathcal{C}}

\begin{document}

\title  {Remarks on the outer length billiards}

\author{Misha Bialy}
\address{School of Mathematical Sciences, Raymond and Beverly Sackler Faculty of Exact Sciences, Tel Aviv University,
	Israel} 
\email{bialy@tauex.tau.ac.il}
\thanks{MB was partially supported by ISF grant 974/24.}
 \author{Serge Tabachnikov}
 \address{Department of Mathematics,
 	Penn State University,
 	University Park, PA 16802, USA}
 \email{tabachni@math.psu.edu}
 \thanks{ST was supported by NSF grants DMS-2404535 and Simons Foundation grant MPS-TSM-00007747.}

\date{\today}

\begin{abstract}
We study outer length billiards; our main results are as follows.
We prove  3- and 4-periodic versions of the Ivrii conjecture. We show that, for every period $n\ge 3$, there exists a functional space of billiard tables that possess invariant curves consisting of $n$-periodic points. For $n=4$, we explicitly parameterize such centrally symmetric billiard tables by functions of one variable and describe how to construct these tables geometrically, similarly to the known construction of Radon curves. 
\end{abstract}

\maketitle

\section{Introduction} \label{sect:int}

This paper continues the study of the outer length billiards,  started in \cite{ACT,AT,BBF}. This discrete-time dynamical system is associated with a plane oval (a closed strictly convex smooth curve) and acts on its exterior.  The precise definition of this billiard is given below,  in Section \ref{sect:form}; here we just say that its $n$-periodic orbit is an $n$-gon, circumscribed about the oval and having an extremal perimeter. 

This system is colloquially called ``the fourth billiard", the reason being that there are three other, better known, billiard models associated with plane ovals. These are the inner length, a.k.a. Birkhoff, billiards (the periodic orbits are the inscribed polygons of  extremal perimeter), the outer billiards (the periodic orbits are the circumscribed polygons of  extremal area), and the symplectic billiards (the periodic orbits are the inscribed polygons of extremal area); see \cite{AT,BBS} for unifying points of view. 

In the present paper we consider two billiard-related problems: the Ivrii conjecture, and the existence and description of invariant curves entirely consisting of periodic points. Both problems have been investigated for the above mentioned other three billiard models.

The outer length billiard transformation is an area preserving twist map,  the invariant area form depending on the billiard table (see Section \ref{sect:form}). The relevant version of the Ivrii conjecture states that the set of periodic points is a null set (a slightly weaker formulation is that this set has the empty interior). 

We show that the second iteration of the outer length billiard map is also a twist map and, using this result, in Theorem \ref{thm:34} we prove the Ivrii conjecture for periods $3$ and $4$. In Section \ref{sect:Ivr} we provide two other proofs of the 3-periodic case, using different approaches.  See \cite{AT1,BKNZ,GT,GlK,Ry,Sh,St,TZ,Tu,Vo,Wo} for similar results concerning the other three billiards. 

A billiard is completely integrable if its phase cylinder, or at least a part of it, is foliated by invariant curves. For all four billiard models, this is the case when the billiard table is an ellipse and, conjecturally, this is also the ``only if" statement, known as Birkhoff's conjecture, see \cite{BM-survey}. In particular, an integrable billiard has  invariant curves consisting of periodic points with all periods $n\ge 3$. The second problem is to construct billiard tables that possess invariant curves consisting of periodic points with a  given period.

In Section \ref{sect:dis}, we show that, for every $n\ge 3$, there is a functional space of outer length billiard tables having invariant curves comprising $n$-periodic point. We use the sub-Riemannian geometric approach that was put forward in \cite{BZ,La} in the case of Birkhoff billiards and applied to outer and symplectic  billiards in \cite{AT1,GT}. We also use this approach in one of the proofs of the 3-periodic Ivrii conjecture in Section \ref{sect:Ivr}.

 Section \ref{sect:4} concerns centrally symmetric outer length billiard tables with invariant curves consisting of 4-periodic points. We show that these periodic quadrilaterals are parallelograms and, in Theorem \ref{thm:4per}, we explicitly parameterize the respective billiard tables by functions of one variable. We also describe a construction of such ovals that resembles the construction of Radon curves (see, e.g., \cite{MS}). Radon curves solve an analogous problem for outer and symplectic billiards \cite{AT1,BBT}.

\medskip

{\bf Acknowledgements}. We thank V. Zharnitsky for helpful discussions.
ST is grateful to the Institute of Advanced Studies of the Tel Aviv University for its support and hospitality. 

\section{Invariant area form, twist condition} \label{sect:form}
The outer length billiard map, denoted by $ T$,  acts in the exterior of the plane oval $\g$
 according to the following rule (see Figure \ref{fig:fourth-rule}).
\begin{figure}[h]
	\centering
	\includegraphics[width=0.75\linewidth]{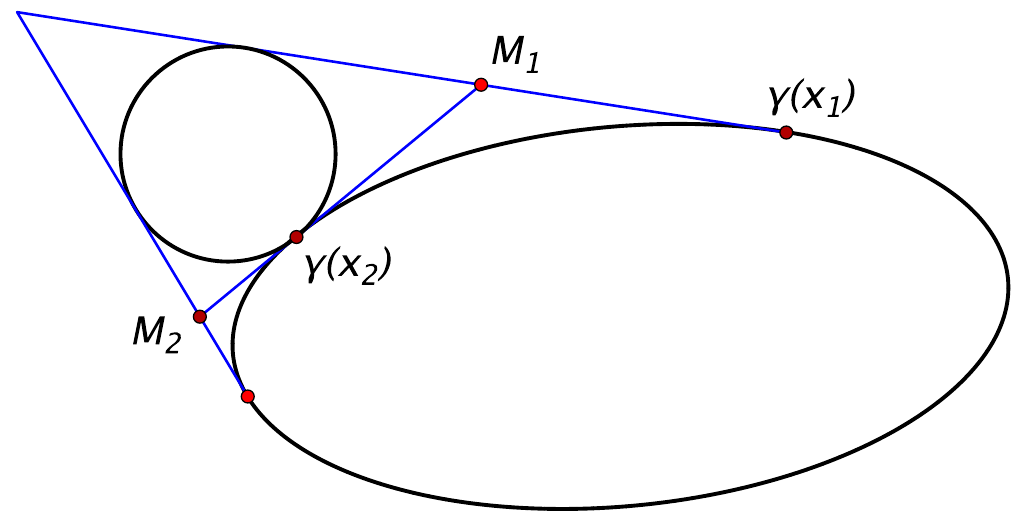}
	\caption{Outer length billiard rule.}
	\label{fig:fourth-rule}
\end{figure}

Given a point $M_1$ in the exterior of $\g(x)$, consider the two tangent lines  from $M_1$ to $\gamma$ and the unique circle tangent to $\gamma$ at point $\gamma(x_2)$ and to the tangent line at $\gamma(x_1)$. Then the point 
$M_2=T(M_1)$ is defined as the intersection point of the tangent line at $\gamma(x_2)$ and the unique line which is tangent to the circle and $\gamma$. We will denote the lengths of the tangent segments as follows:
$$
l_1:=|M_1-\gamma(x_1)|,\quad l_2:=|M_1-\gamma(x_2)|.
$$

One can prove \cite{AT} that $ T$ is a {positive} twist map, and that $S=l_1+l_2-\widearc{\g(x_1)\g(x_2)}$ is a generating function for the outer length billiard.

We will use the envelope coordinates of $\g$  to compute 
the generating function $S=l_1+l_2-\widearc{\g(x_1)\g(x_2)}$.
\begin{figure}[h]\label{lambda}
	\centering
	\includegraphics[width=0.73\linewidth]{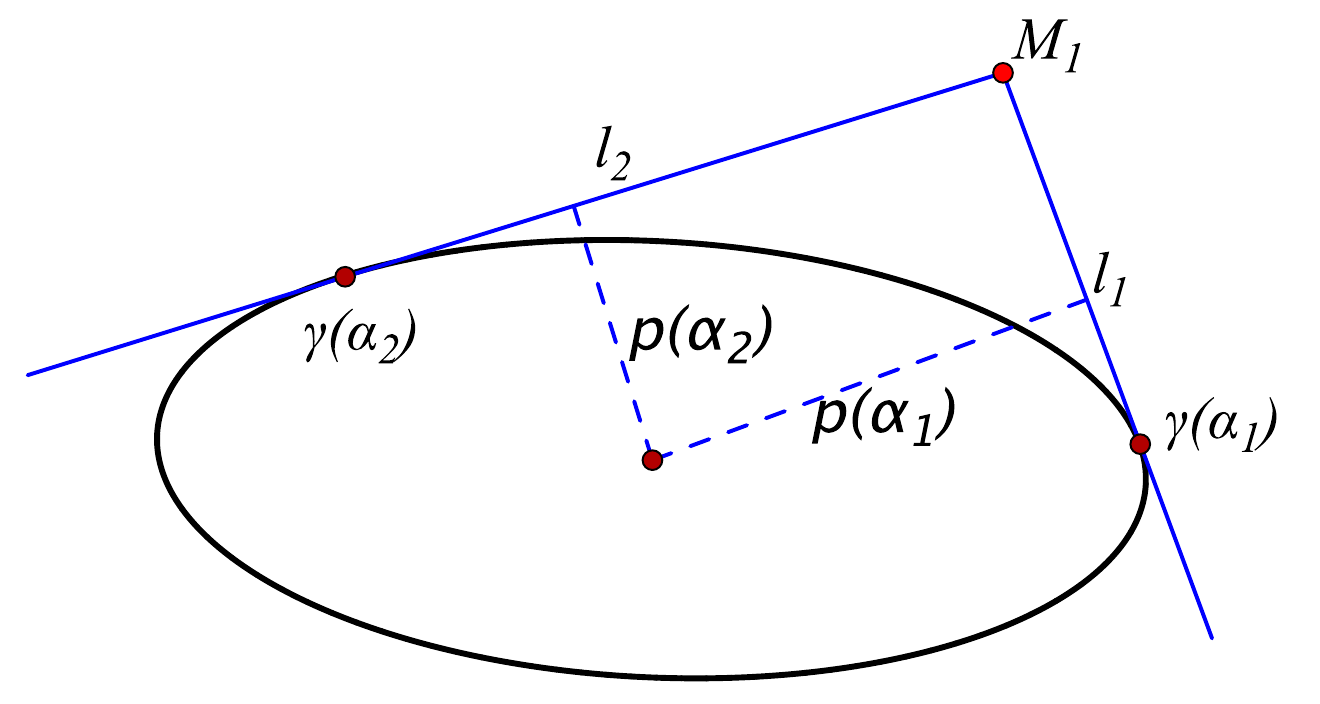}
	\caption{Computing $l_1, l_2$ in envelope coordinates.}
\end{figure}

Denote by $p$ the support function of $\g$ (see Figure \ref{lambda}).
In particular:
\begin{equation}
	\gamma (\alpha) = p(\alpha)\, (\cos \alpha, \sin \alpha) + p'(\alpha)\, (- \sin \alpha, \cos \alpha).\\
	\label{eq:envelope}
\end{equation}
 We have the following formulae for $l_1,l_2$, and $S$ (see Lemma 4.4 of \cite{BBS}).
 
\begin{lemma}\label{lemma:lambdas} 
One has
\begin{equation*}
\begin{aligned}	
&l_1=-p'(\alpha_1)+\frac{p(\alpha_2)}{\sin(\alpha_2-\alpha_1)}-
p(\alpha_1)\cot(\alpha_2-\alpha_1),\\
&l_2=p'(\alpha_2)+\frac{p(\alpha_1)}{\sin(\alpha_2-\alpha_1)}-p(\alpha_2)\cot(\alpha_2-\alpha_1),\\
&S(\alpha_1,\alpha_2)=l_1+l_2-\widearc{\g(x_1)\g(x_2)}=\bigg(p(\alpha_2)+p(\alpha_1)\bigg)\tan\left(\frac{\alpha_2-\alpha_1}{2}\right)-\int_{\alpha_1}^{\alpha_2}p d\alpha.
\end{aligned}
\end{equation*}
\end{lemma}

\proof Since $p''+p$ is the curvature radius of $\g$,
the arc length is given by 
$$
\widearc{\g(x_1)\g(x_2)}=\int_{\alpha_1}^{\alpha_2}(p''+p) d\alpha=p'(\alpha_2)-p'(\alpha_1)+\int_{\alpha_1}^{\alpha_2}p d\alpha.
$$ 
Using this fact and the first two formulas of the lemma, we obtain the expression for $S$.
	
	Now we compute $l_{1}, l_{2}$.
The coordinates of the points $\gamma(\alpha_1)$ and $\gamma(\alpha_2)$ can be found using formula (\ref{eq:envelope}), yielding 
	\begin{align*}
		\gamma (\alpha_1) &= p(\alpha_1)\, (\cos \alpha_1, \sin \alpha_1) + p'(\alpha_1)\, (- \sin \alpha_1, \cos \alpha_1),\\
		\gamma (\alpha_2) &= p(\alpha_2)\, (\cos \alpha_2, \sin \alpha_2) + p'(\alpha_2)\, (- \sin \alpha_2, \cos \alpha_2).
	\end{align*}
	
	Next, one needs to find the coordinates of the point $M$ as the intersection of the tangent lines. Thus  $(x_M,y_M)$, the coordinates of $M$, satisfy the system:
	\[
	\begin{cases}
		\cos\alpha_1 x+\sin\alpha_1 y=p(\alpha_1) \\
		\cos\alpha_2 x+\sin\alpha_2 y=p(\alpha_2).
	\end{cases}
	\]
	
	Solving this system, we get:
	$$
	(x_M,y_M)=\frac{1}{\sin(\alpha_2-\alpha_1)}\bigg(p(\alpha_1)\sin\alpha_2-
p(\alpha_2)\sin\alpha_1,\ p(\alpha_2)\cos\alpha_1-
	p(\alpha_1)\cos\alpha_2\bigg).
	$$
	Hence
	
	\begin{eqnarray*}
		l_1 &=& \frac{x_M-x_{\gamma(\alpha_1)}}{\cos(\alpha_1+\pi/2)} \\
		&=& - \frac{p(\alpha_1)\sin\alpha_2-
			p(\alpha_2)\sin\alpha_1-\sin(\alpha_2-\alpha_1)(p(\alpha_1) \cos \alpha_1- p'(\alpha_1)\sin \alpha_1)}{\sin\alpha_1{\sin(\alpha_2-\alpha_1)}}.
	\end{eqnarray*}
	Substituting into the numerator of this formula the identity
	$$\sin\alpha_2=\sin\alpha_1\cos(\alpha_2-\alpha_1)+\cos\alpha_1\sin(\alpha_2-\alpha_1),$$ 
	we obtain the first claim of  Lemma \ref {lemma:lambdas}.
	The second formula  is proved similalrly.
	\proofend

Now we are in position to compute the partial derivatives of the function $S$, see Figure \ref{RR}.

\begin{lemma}\label{lemma:partialsS}
The partial derivatives of $S(\alpha_1,\alpha_2)$ with respect to the first and the second arguments are
\begin{align*}
S_{1} =\frac{1}{2\cos^2 \omega/2}( p(\alpha_1)\cos\omega-p(\alpha_2)+ p'(\alpha_1)\sin\omega)=-l_1 \tan\frac \omega 2=-R_1,\\
S_{2} =\frac{1}{2\cos^2 \omega/2}(- p(\alpha_2)\cos\omega+p(\alpha_1)+ p'(\alpha_2)\sin\omega)=l_2 \tan\frac \omega 2=R_2,
\end{align*}
where $\omega:=\alpha_2-\alpha_1$, and $R_1,R_2$ are the radii of the auxiliary circles.
\end{lemma}

\begin{corollary} One has
	$dS=R_2d\alpha_2-R_1d\alpha_1$, so $dR\wedge d\alpha$ is an invariant symplectic form.
\end{corollary}

\begin{figure}[h] 
	\centering
	\includegraphics[width=0.75\linewidth]{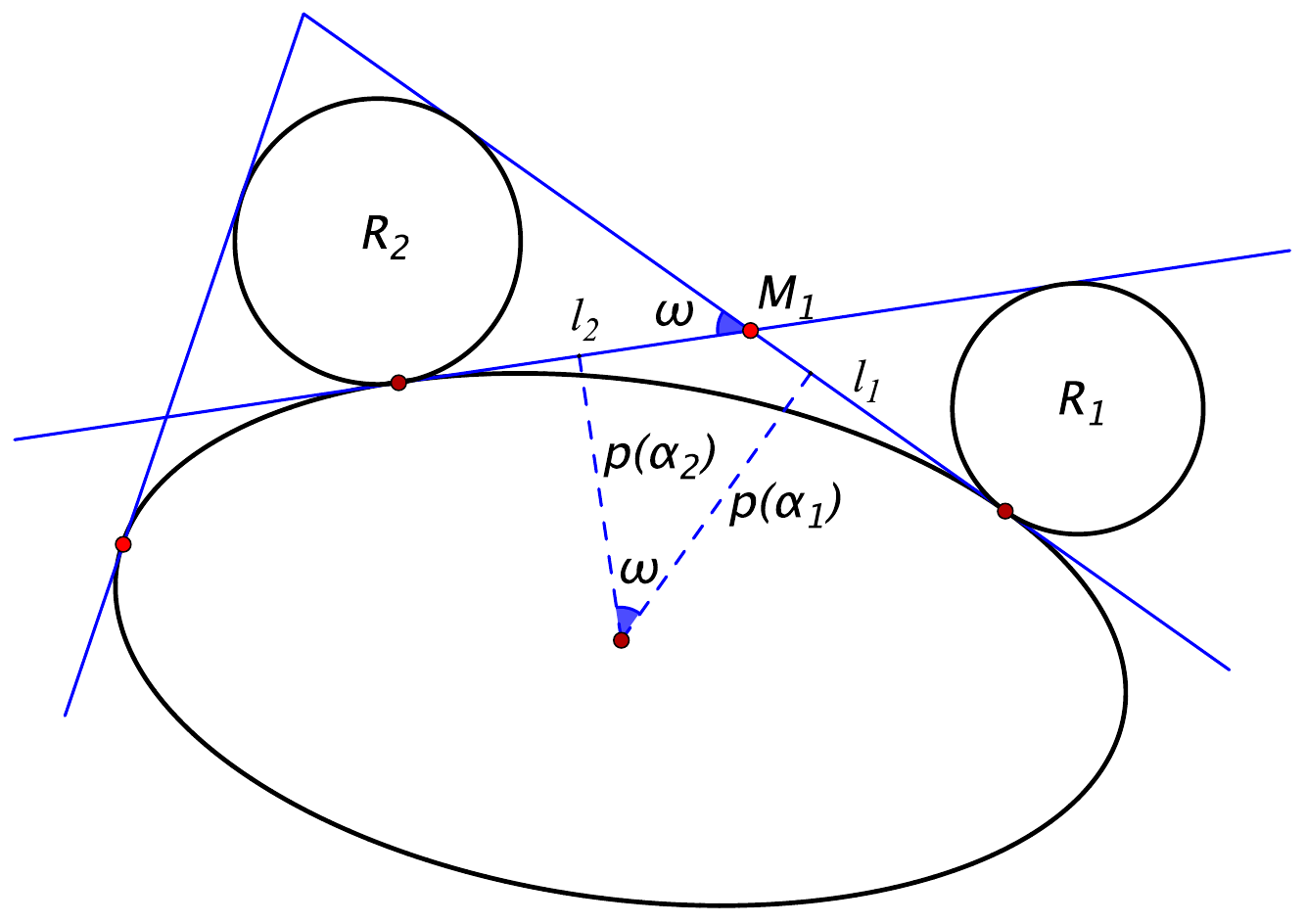}
	\caption{Invariant symplectic form $dR\wedge d\alpha$.}
	\label{RR}
\end{figure}

Next we need to compute second partial derivatives of $S$.
\begin{lemma}
	The second partial derivatives of $S$ are as follows:
\begin{align*}
	&S_{11}=\tan \frac \omega 2(R_1+\rho (\alpha_1)),\\
	&S_{22}=\tan \frac \omega 2(R_2+\rho (\alpha_2)),\\
	&S_{12}=-\frac 1 {\sin\omega }(R_1+R_2), 
\end{align*}
where $\omega=\alpha_2-\alpha_1$ and $\rho$ is the radius of curvature of $\g$.
	In particular $S_{12}<0$ and $S_{11}, S_{22}>0$.
\end{lemma}
\begin{corollary}
	The maps $T$ and $T^2$ are positive twist maps of the phase cylinder with respect to the coordinates $R, \alpha$.
\end{corollary}

\proof
	The differential  of $T:(R,\alpha)\mapsto(R',\alpha')$ can be written in the form:
	$$
	DT=\begin{pmatrix}
		-\frac{S_{22}}{S_{12}} & \frac{S_{12}^2-S_{11}S_{22}}{S_{12}}\\
		-\frac{1}{S_{12}}&-\frac{S_{11}}{S_{12}} 
	\end{pmatrix}.
	$$
	Thus $\frac{\partial \alpha'}{\partial R}=-\frac{1}{S_{12}}>0$. 
	Moreover, since $S_{11} $ and $S_{22}$ are strictly positive, we see that  the product of two matrices $DT(R,\alpha)\cdot DT(R',\alpha')$ also has a positive entry $\frac{\partial \alpha''}{\partial R}>0$. Hence $T^2$ is  a positive twist map as well. 
	\proofend
	
	As an application of the twist map property for $T,T^2$, we have
	
	\begin{theorem} \label{thm:34}
		The sets of 3-periodic orbits and 4-periodic points of the outer length billiard have empty interior. 
	\end{theorem}
	
\proof
We shall give a proof for the set of 4-periodic orbits.
	The other case is analogous (and we provide two other proofs in the 3-periodic case later).
	
Arguing toward  contradiction, assume that  there is a disc of 4-periodic points. Choose a 
	point $P_0$  in the disc and let $v_0=\frac{\partial}{\partial R}$ be the vertical tangent vector. Define $v_i:=DT^i(v_0), i=1,2,3.$ By 4-periodicity, we have   $$ DT^4(v_0)=v_0.$$
	Let $z_i, i=0,\ldots,3$, be the $\alpha$-coordinates of the vectors $v_i$. Then $z_0=0$, and
	since $T,T^2$ are positive twist maps, we have that $z_1,z_2$ are strictly positive.  On the other hand,
	the maps $T^{-1},T^{-2}$ are negative twist maps (as the inverses of  positive twists), and hence
	 $z_3$ and $z_2$ are strictly negative, since $DT^{-1}v_0=v_3, DT^{-2}v_0=v_2 $. This is a contradiction.
\proofend
	
	\begin{remark}
	{\rm	
		1. Analogously one can prove that if $T,T^2,..,T^k, k\geq 1$ are positive twist maps, then the sets of periodic orbits of periods $2,3,\ldots,2k$ have empty interior.
		
		2. If all  positive powers of $T$
		are positive twist maps, then $T$ is totally integrable.  This situation is proved to be very rigid for many billiard models, see \cite{BM-survey}.
		}
	\end{remark}
	
\section{Periodic invariant curves and a  distribution on the space of polygons} \label{sect:dis}

Our goal in this section is to show that, for every $n\ge3$, there exists a functional space of outer length billiard curves that possess invariant curves consisting of $n$-periodic points. Our approach is via sub-Riemannian geometry, following \cite{BZ,La,GT}.

Let $\g$ be an oval, the outer length billiard table, and assume that there exists an invariant curve consisting of $n$-periodic points. 

\begin{figure}[ht]
\centering
\includegraphics[width=2.7in]{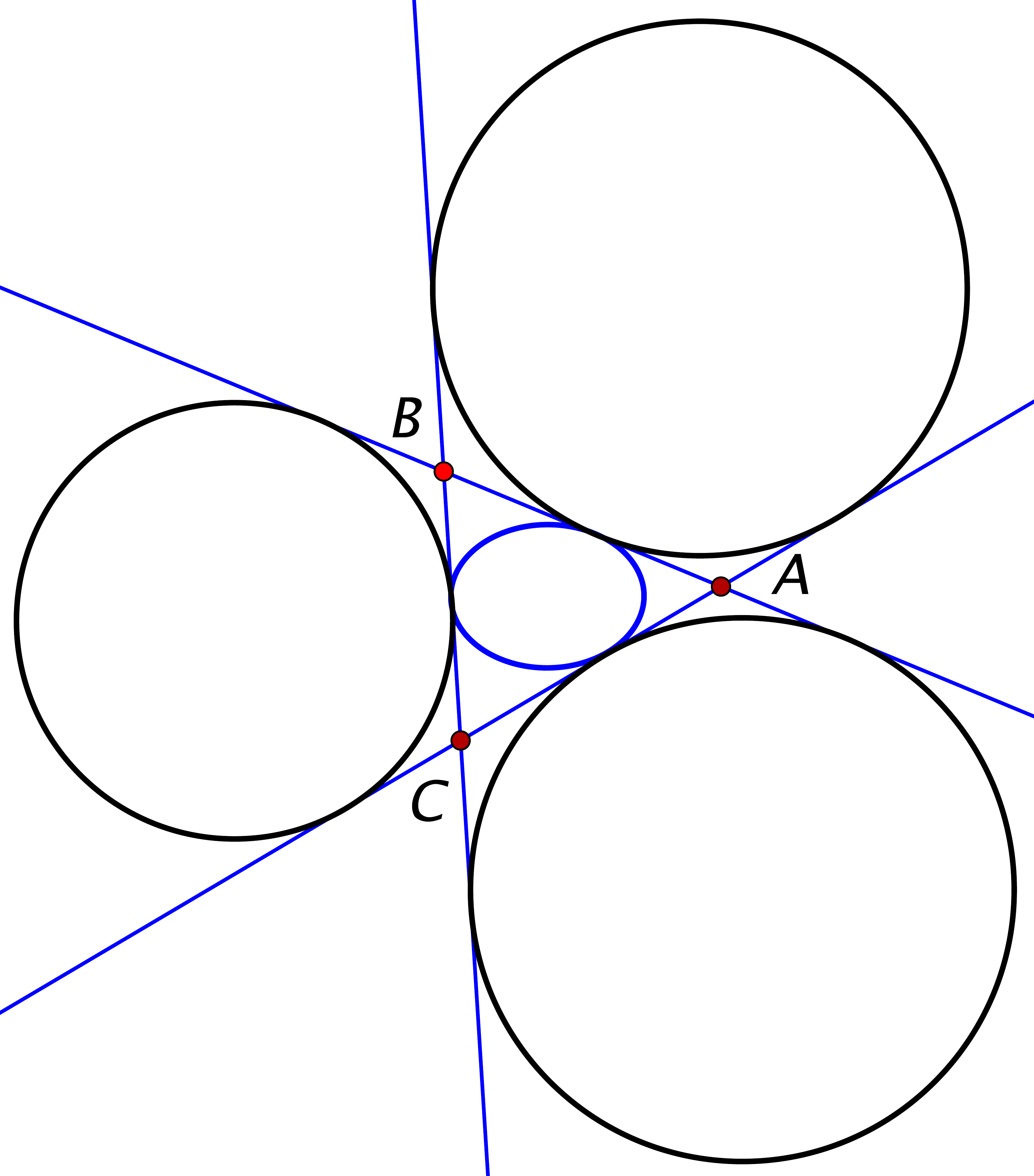}
\caption{$ABC$ is a 3-periodic orbit of the outer length billiard.}
\label{inv}
\end{figure}

For example, consider Figure \ref{inv}. If the 3-periodic orbit $ABC$ belongs to a 1-parameter family of such orbits then the sides of these triangles envelop the oval $\g$. That is, the  motion of each side is an infinitesimal rotation about the tangency point of the respective circle with this side. 

This motivates the following definitions. 

 Let $\P_n$ be the space of convex $n$-gons in the plane, a $2n$-dimensional manifold. 
 We think of a polygon as a cyclically ordered set of its sides $L_i,\ i=1,\ldots,n$, cooriented by the external normals. 
 For the vertices of the polygon, we use the notation $V_{i+\frac12} := L_i \cap L_{i+1}$.
 The vertices and the sides are cyclically ordered counterclockwise. 
 
 Choose an origin, and let $(\alpha_i,p_i)$ be the support coordinates of $L_i$, that is, $\alpha_i$ is the direction of the coorienting normal and $p_i$ is the signed distance from $L_i$ to the origin (if the origin is inside the polygon, all distances are positive).
  Then the unit coorienting vector is $(\cos\alpha_i,\sin\alpha_i)$, and the equation of the line in Cartesian coordinates is $x\cos\alpha_i+y\sin\alpha_i=p_i$. 
 
 Define an $n$-dimensional distribution $\D_n$ on ${\mathcal P}_n$.  Consider three consecutive lines $L_{i-1}, L_i, L_{i+1}$. Consider the circle that is tangent to these lines (there are four such circles) and lies on the negative sides of the cooriented lines $L_{i-1}$ and $L_{i+1}$ and on the positive side of the line $L_i$. Let $C_i$ be the tangency point of this circle with the line $L_i$, see Figure \ref{defd}.  Let $\xi_i$ be the infinitesimal rotation of the line $L_i$ about point $C_i$. The distribution $\D_n$ is spanned by the vector fields $\xi_i,\ i=1,\ldots, n$. 
  
 \begin{figure}[ht]
\centering
\includegraphics[width=2.5in]{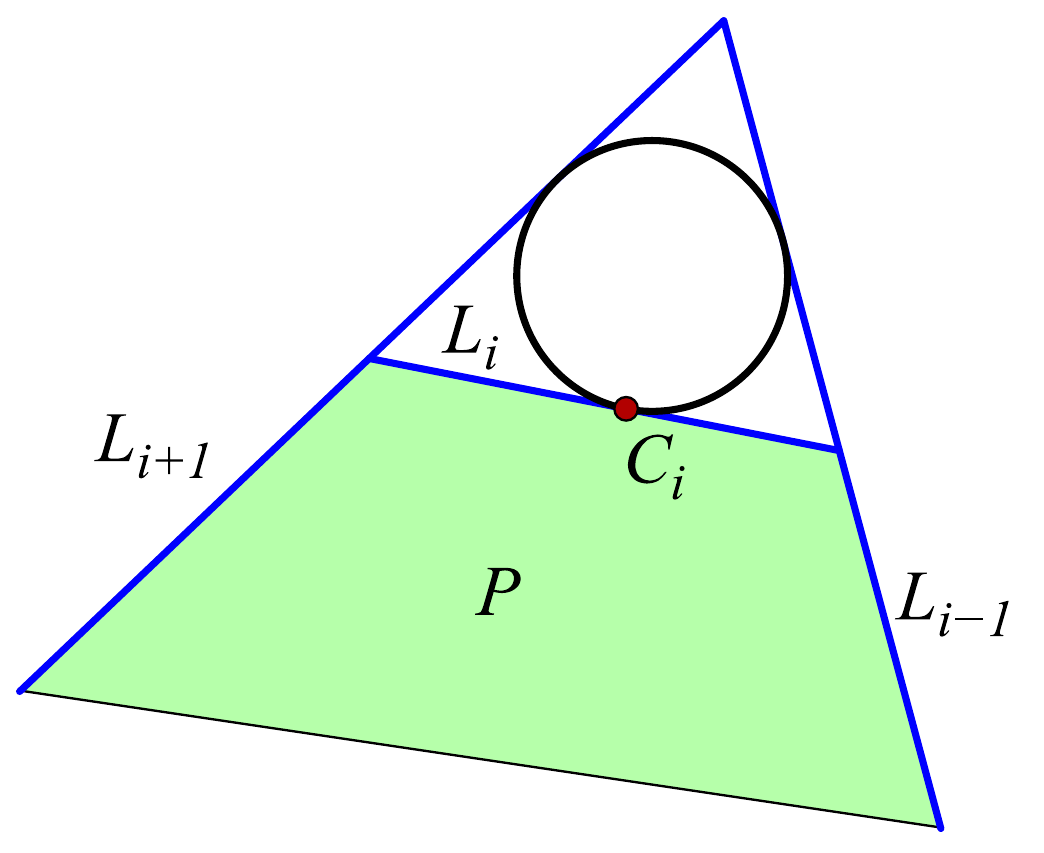}
\caption{Definition of the distribution $\D_n$.}
\label{defd}
\end{figure}
 
 Let $F: \P_n \to \R$ be the perimeter of a polygon.
 
 \begin{lemma} \label{lm:per}
 The distribution $\D_n$ is tangent to the level hypersurfaces of the function $F$ on $\P_n$.
 \end{lemma}
 
 \proof The argument is essentially the proof of Lemma 3.1 in \cite{AT}.  
 
 Consider Figure \ref{extrem}. Fix point $C_i$ and let 
 the side $L_i$, passing through $C_i$, be variable. We claim that the perimeter of the polygon is extremal when $Z_i=C_i$,
that is, the circle passes through point $C_I$.
 
 \begin{figure}[ht]
\centering
\includegraphics[width=4in]{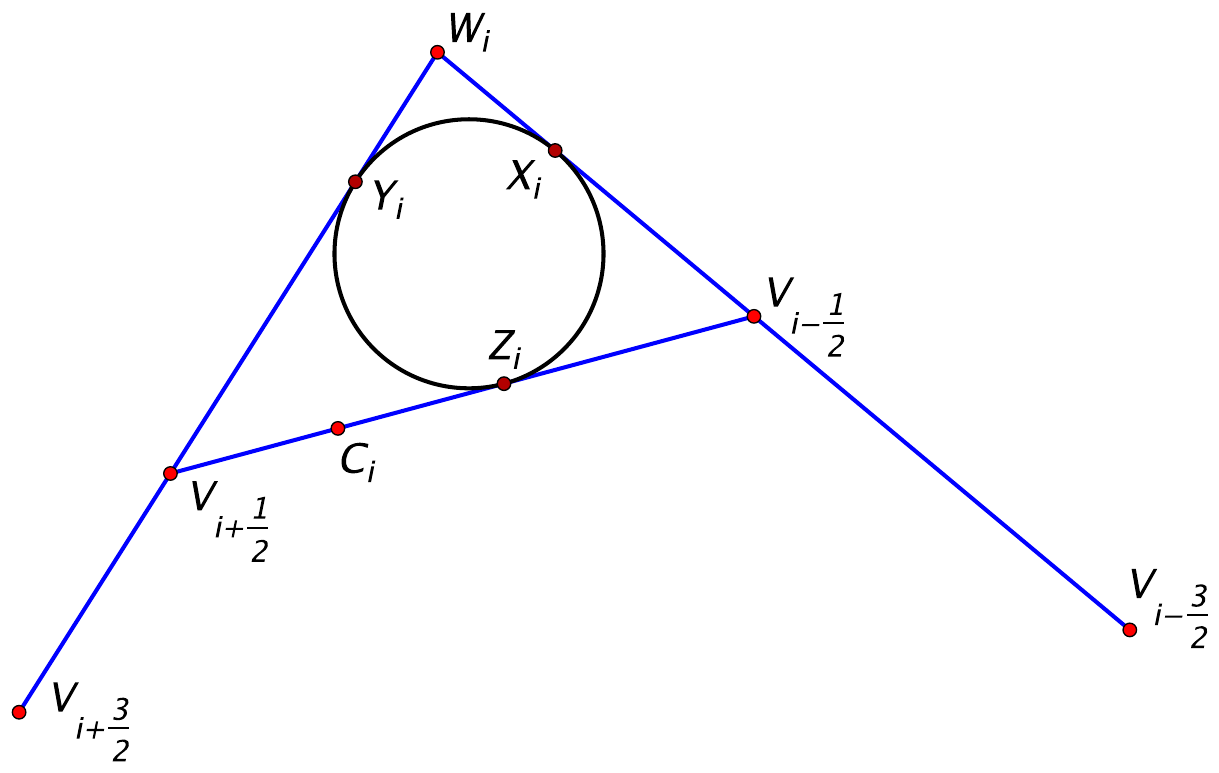}
\caption{Proof of Lemma \ref{lm:per}. If point $W_i$ lies on the other side of the line $L_i$ or the sides $L_{i-1}$ and $L_{i+1}$ are parallel, the argument is similar. }
\label{extrem}
\end{figure}
 Indeed, the two tangent segments to a circle from a point have equal lengths. Hence 
 $  |V_{i-\frac12} Z_i|=|V_{i-\frac12} X_i|, |V_{i+\frac12} Z_i|=|V_{i+\frac12} Y_i|$, and 
 $$
 \begin{aligned}
 |V_{i-\frac32} V_{i-\frac12}|+|V_{i-\frac12} V_{i+\frac12}|+|V_{i+\frac12} V_{i+\frac32}|&=|V_{i-\frac32} X_i|+|V_{i+\frac32} Y_i|\\
 &=|V_{i-\frac32} W_i |+|V_{i+\frac32}W_i|-2|W_iX_i|.
 \end{aligned}
 $$
 It follows that the perimeter is extremal when so is the tangent segment $W_iX_i$, and this happens when the circle passes through point $C_i$.
 
 Returning to Figure \ref{defd}, it follows that one has, for the directional derivative,  $D_{\xi_i} (F)=0$, as needed.
 \proofend
 
In view of the above lemma, we restrict attention to polygons with a fixed perimeter, say, unit perimeter. Let $\C_n$ be the set of such convex $n$-gons. Then dim $\C_n = 2n-1$ and $\D_n$ is an $n$-dimensional distribution on $\C_n$.
Here is the main result of this section.

\begin{theorem} \label{thm:nper}
For every $n\ge 3$, there exists a functional space of  outer length billiard curves, sufficiently close to a circular one and possessing invariant curves consisting of $n$-periodic orbits.
\end{theorem}

Here is a description, in broad strokes, of the proof.  We construct outer length billiard curves comprising $n$-periodic orbits by perturbing a circle (which obviously has the said  property). 
The circular $n$-periodic invariant curve of the outer length billiard is a closed horizontal curve $\g_0$ in the space $\C_n$. One has the action of the  group $\Z_n$ on $\C_n$ that cyclically permutes the sides of $n$-gons and that preserves the distribution $\D_n$, and $\g_0$ is invariant under this cyclic permutation. We show that $\g_0$ admits a functional space of perturbations in the class of horizontal closed $\Z_n$-invariant curves in $\C_n$. The sides of the respective triangles then envelop the desired outer length billiard curves. We refer to \cite{Mon,Mon2} for the underlining theory.

As a preparation to the proof, we make some computations.

Let $X=(a,b)$ be Cartesian coordinates of a point, and let $x\cos\alpha+y\sin\alpha=p$ be the equation of a line through   point $X$. The counterclockwise infinitesimal rotation of the line about point $X$ is a tangent vector $\xi$ to the space of lines.

\begin{lemma} \label{lm:rot}
One has 
$$
\xi= \partial \alpha + (b\cos\alpha - a\sin\alpha)\ \partial p = \partial \alpha + [(\cos\alpha,\sin\alpha),(a,b)]\ \partial p,
$$
where the bracket denotes the determinant of a pair of vectors.
\end{lemma}

\proof
The equation of the lines through point $X$ is $p(\alpha)=a\cos\alpha+b\sin\alpha$. Differentiating with respect to $\alpha$ yields the result.
\proofend

Let $2\beta_{i-\frac12}:=\alpha_i-\alpha_{i-1}$ be  the exterior angle at vertex $V_{i-\frac12}$. Let us calculate the coordinates of the vertices $V_j$ (where $j$ is a half-integer) and the infinitesimal rotations of the $i$th sides about points $C_i$.

\begin{lemma} \label{lm:coord}
One has:
$$
V_{i-\frac12}=\frac{(p_{i-1} \sin \alpha_i-p_i \sin \alpha_{i-1},
p_{i} \cos \alpha_{i-1}-p_{i-1} \cos \alpha_i)}{\sin (\alpha_i-\alpha_{i-1})},
$$
and 
\begin{equation} \label{eq:cr}
[(\cos\alpha_i,\sin\alpha_i),C_i] = 
\frac{\cos^2 \left(\frac{\alpha_i - \alpha_{i-1}}{2}\right) (p_{i+1}+p_i) - \cos^2 \left(\frac{\alpha_{i+1} - \alpha_{i}}{2}\right) (p_{i-1}+p_i) }
{2\sin \left(\frac{\alpha_{i+1}-\alpha_{i-1}}{2}\right) \cos \left(\frac{\alpha_i - \alpha_{i-1}}{2}\right) \cos \left(\frac{\alpha_{i+1} - \alpha_{i}}{2}\right)}.
\end{equation}
\end{lemma}

\proof
Point $V_{i-\frac12}$ is the intersection of the lines given by the equations $x\cos\alpha_{i-1}+y\sin\alpha_{i-1}=p_{i-1}$ and $x\cos\alpha_i+y\sin\alpha_i=p_i$. Solving this system of linear equations for $x,y$ and using some trigonometry yields the result (see  the proof of Lemma \ref{lemma:lambdas}).

Let $r_i$ be the radius of the circle in Figure \ref{defd}. Then
$$
|V_{i-\frac12} C_i|=r_i \cot \beta_{i-\frac12},\ |V_{i+\frac12} C_i|=r_i \cot \beta_{i+\frac12},
$$
hence 
$$
C_i = \frac{\cot \beta_{i-\frac12} V_{i+\frac12} + \cot \beta_{i+\frac12} V_{i-\frac12}}{\cot \beta_{i-\frac12}+\cot \beta_{i+\frac12}}.
$$

Next, using the obtained formulas for $C_i$ and $V_{i\pm\frac12}$, we calculate:
$$
[(\cos\alpha_i,\sin\alpha_i),C_i] = \frac{(\cos^2 \beta_{i-\frac12}) p_{i+1} - (\cos^2 \beta_{i+\frac12}) p_{i-1} - (\cos^2 \beta_{i+\frac12}-\cos^2 \beta_{i-\frac12}) p_i}
{2\sin(\beta_{i-\frac12}+\beta_{i+\frac12})\cos \beta_{i-\frac12} \cos \beta_{i+\frac12}}.
$$
This is equivalent to formula (\ref{eq:cr}). 
\proofend 

\begin{remark} \label{rmk:per}
{\rm The following formulas are worth mentioning: the length of the $i$th side of the polygon equals
$$
\frac{p_{i-1}}{\sin(\alpha_i-\alpha_{i-1})} + \frac{p_{i+1}}{\sin(\alpha_{i+1}-\alpha_{i})} - \frac{p_i \sin(\alpha_{i+1}-\alpha_{i-1})}{\sin(\alpha_i-\alpha_{i-1})\sin(\alpha_{i+1}-\alpha_{i})},
$$
and the perimeter $F$ is given by 
$$
\sum p_i \left[\tan\left(\frac{\alpha_{i+1}-\alpha_i}{2}\right) +  \tan\left(\frac{\alpha_{i}-\alpha_{i-1}}{2}\right) \right] =
\sum (p_i+p_{i+1}) \tan\left(\frac{\alpha_{i+1}-\alpha_i}{2}\right).
$$
One can directly verify that the directional derivatives of the perimeter with respect to the vectors $\xi_i$ vanish, that is, to reprove Lemma \ref{lm:per} analytically.
}
\end{remark}

Denote by $\Phi_i$ the right hand side expression in (\ref{eq:cr}). Then the distribution $\D_n$ is generated by the vector fields $\xi_i = \partial \alpha_i + \Phi_i \partial p_i$. Since $\Phi_i$ depends only on the variables with the indices $i-1,i,i+1$, one has  $[\xi_i,\xi_j]=0$ for $|i-j| \ge 2$.

\begin{proposition} \label{prop:nonint}
In a sufficiently small neighborhood of the curve $\g_0$, the distribution $\D_n$ is completely non-integrable. More specifically, it has the growth type $(n,2n-1)$, that is, the tangent space $T \C_n$ is spanned by the vector fields $\xi_i$ and their first commutators.
\end{proposition}

\proof One has
$$
[\xi_i,\xi_{i+1}]= \left(\frac{\partial\Phi_{i+1}}{\partial \alpha_i} +\Phi_i \frac{\partial\Phi_{i+1}}{\partial p_i}\right) \partial p_{i+1}
- \left(\frac{\partial\Phi_{i}}{\partial \alpha_{i+1}} +\Phi_{i+1} \frac{\partial\Phi_{i}}{\partial p_{i+1}}\right) \partial p_{i},
$$
where the indices are understood cyclically. Let us evaluate these brackets along the curve $\g_0$.

Choose the origin at the center of the circle; then, after rescaling, one has $p_i=1$ and $\alpha_{i+1}-\alpha_i= \frac{2\pi}{n}$ for all $i$. It follows that $\Phi_i=0$ for all $i$, hence 
$$
[\xi_i,\xi_{i+1}]=\frac{\partial\Phi_{i+1}}{\partial \alpha_i} \partial p_{i+1} - \frac{\partial\Phi_{i}}{\partial \alpha_{i+1}} \partial p_{i}.
$$
along $\g_0$.

Next, the operation of setting $p_j=1$ for all $j$ and differentiating with respect to $\alpha_i$ commute. Setting all $p_j=1$  and using some trigonometry, one has 
\begin{equation} \label{eq:Phitan}
\begin{aligned}
\Phi_{i+1} &= \frac{\cos^2 \beta_{i+\frac12}-\cos^2 \beta_{i+\frac32}}{\sin(\beta_{i+\frac12}+\beta_{i+\frac32}) \cos \beta_{i+\frac12} \cos \beta_{i+\frac32}}
= \frac{\sin\left(\beta_{i+\frac32}-\beta_{i+\frac12}\right)}{\cos \beta_{i+\frac12} \cos \beta_{i+\frac32}}\\ 
&= \tan\left(\frac{\alpha_{i+2}-\alpha_{i+1}}{2}\right) - \tan\left(\frac{\alpha_{i+1}-\alpha_{i}}{2}\right),
\end{aligned}
\end{equation}
hence, along $\g_0$, 
$$
\frac{\partial\Phi_{i+1}}{\partial \alpha_i} = \frac{1}{2\cos^2 \beta_{i+\frac12}} = \frac{1}{2 \cos^2 \frac{\pi}{n}}>0. 
$$
Likewise, $\frac{\partial\Phi_{i}}{\partial \alpha_{i+1}} > 0.$

Let $\theta_i= dp_i-\Phi_i d\alpha_i,\ i=1,\ldots,n$, be the 1-forms that span the conormal bundle $\D^\perp$ of the distribution $\D$. It follows that, along $\g_0$, the matrix $\theta_j([\xi_i,\xi_{i+1}]), j=1,\ldots,n-1, i=1,\ldots,n$, has rank $n-1$. Therefore this rank equals $n-1$ is a sufficiently small neighborhood of $\g_0$ as well, that is, $T\C_n=\D+[\D,\D]$, as claimed.
\proofend

\begin{remark} 
{\rm We believe that $\D_n$ has the growth type $(n,2n-1)$ everywhere in $\C_n$, and we checked it for $n=3$, but, at the moment, do not have a  proof of this conjecture.
}
\end{remark}

\noindent {\bf Proof of Theorem \ref{thm:nper}.}
A circular outer length billiard yields a 1-parameter family of periodic $n$-gons $L(t)=(L_1(t),\ldots,L_n(t)),\  t\in [0,1]$, circumscribed about the circular table (we continue to think of a polygon as a tuple of its sides). Let $\sigma$ be the cyclic perturbation of the sides; then $L(1)=\sigma(L(0))$. This family of $n$-gons is the initial horizontal curve $\gamma_0: [0,1] \to \C_n$, given, after rescaling, by
$$
 \alpha_i(t)=\frac{2\pi}{n}(t+(i-1)),\ p_i(t)=1,\ i=1,\ldots,n.
$$

Denote by $H^2_{\D_n}$ the space of horizontal curves $\g: [0,1]\to \C_n$ whose first two derivatives are square integrable, and let $H^2_{\D_n}(c)$ denote the subspace of paths that start at $c \in \C_n$. Then $H^2_{\D_n}(c)$ is a Hilbert manifold whose local coordinate charts are as follows.

We have the  horizontal vector fields $\xi_i, i=1,\ldots,n$, that span $\D_n$. If $\g(t)$ is a horizontal curve then $\gamma'(t)=\sum_{i=1}^n f_i(t)\xi_i(\g(t))$, where $f_i:[0,1]\to \R$ are square integrable functions. These functions are the coordinates of the curve $\g(t)$. Thus the functional dimension of $H^2_{\D_n}(c)$ is $n$, and the space of non-parameterized oriented horizontal curves with a fixed starting point has functional dimension $n-1$.

Denote by $H^2_{\D_n}(c_0,c_1)\subset H^2_{\D_n}$ the space of the horizontal curves with the starting point $c_0\in\C_n$ and the terminal point $c_1\in\C_n$. The  circular outer length billiard table yields the above mentioned horizontal curve $\g_0 \in H^2_{\D_n}(L(0),L(1))$. We want to perturb this curve within the subspace $H^2_{\D_n}(L(0),L(1))$.

We  show that $H^2_{\D_n}(L(0),L(1)) \subset H^2_{\D_n}(L(0))$ is a Hilbert submanifold of codimension $2n-1$. To this end, let $\pi: H^2_{\D_n}(L(0)) \to \C_n$ be the  projection that assigns to a path $\g$ its endpoint $\g(1)$. If this smooth map is a submersion, then the preimage $\pi^{-1}(L(1))$ is a Hilbert  submanifold of codimension $2n-1$. 

However $\pi$ may fail to be a submersion; the curves for which this happens are called singular. An example of a singular curve is a segment of the $x$-axis in 3-space with the Martinet distribution given by  $dz-y^2dx=0$: the perturbations of this curve as a horizontal curve with fixed endpoints are just its reparameterizations. 

Thus we need to check that $\g_0$ is not singular, and we  use a  criterion due to Hsu \cite{Hsu}, in the form presented in Section 4 of \cite{Mon2}. We formulate this criterion in general form and then apply it in our special case. 

Let $M$ be a smooth manifold with a distribution $\D$, and let $\eta_i$ be a collection of differential 1-forms that constitute a basis of the conormal bundle $\D^\perp$. Then every covector $\D^\perp$ can be written as $\sum \lambda_i \eta_i$. Consider the differential 2-form $\omega(\lambda):= \sum \lambda_i d\eta_i$. A horizontal curve $\g$ is not singular if its tangent vector $\dot \g$ never lies in the kernel of the restriction of $\omega(\lambda)$ to $\D$. 

In our situation, the 1-forms $\theta_i, i=1,\ldots,n$, generate $\D^\perp$, but these  forms are not independent: they satisfy one linear relation. Along the curve $\g_0$, one has $\theta_i=dp_i$. It follows from the formula for the perimeter in Remark \ref{rmk:per} that,  along $\g_0$,  
$$
0=dF = 2\tan \left(\frac{\pi}{n} \right) \sum_{i=1}^n dp_i.
$$
Therefore $\sum_{i=1}^n \theta_i = 0$ along $\g_0$, and to factor this relation out, we assume that $\sum_{i=1}^n \lambda_i =0$. 

Along $\g_0$, one has
$$
\frac{\partial \Phi_i}{\partial p_{i+1}} = - \frac{\partial \Phi_i}{\partial p_{i-1}} = \frac{1}{2\sin \left(\frac{2\pi}{n} \right)},\
\frac{\partial \Phi_i}{\partial p_{i}} = 0,\ \frac{\partial \Phi_i}{\partial \alpha_{i-1}} = \frac{\partial \Phi_i}{\partial \alpha_{i+1}} = \frac{1}{2\cos^2 \left(\frac{\pi}{n} \right)},
$$
hence 
$$
d\theta_i = \frac{1}{2\sin \left(\frac{2\pi}{n} \right)} d\alpha_i \wedge (dp_{i+1}-dp_{i-1}) + \frac{1}{2\cos^2 \left(\frac{\pi}{n} \right)} (d\alpha_i \wedge d\alpha_{i+1} - d\alpha_{i-1} \wedge d\alpha_{i}). 
$$
One also has
$$
\dot \g_0 = \frac{2\pi}{n} \sum_{j=1}^n \partial \alpha_i,
$$
therefore 
$$
\dot \g_0\, \lrcorner\, d\theta_i = \frac{\pi}{2n\cos^2 \left(\frac{\pi}{n} \right)} (d\alpha_{i-1} - 2d\alpha_i + d\alpha_{i+1}).
$$
Finally, assume that $\dot \g_0\, \lrcorner\, \omega(\lambda) = 0$ on $\D_n$, that is, $\omega(\lambda)(\dot \g_0,\xi_i)=0$ for $i=1,\dots,n$. Then $\lambda_{i-1}-2\lambda_i+\lambda_{i+1}=0$ for all $i$,
hence $\lambda_1=\ldots=\lambda_n$. But $\sum_{i+1}^n\lambda_i=0$, therefore $\omega(\lambda)=0$.

We conclude that the curve $\g_0$ is nonsingular, completing the proof.
\proofend

\section{Period four, the centrally symmetric case} \label{sect:4}

 In this section, we provide an explicit parameterization of centrally symmetric outer length billiard curves $\g$ that possess invariant curves consisting of 4-periodic points. 

\begin{lemma} \label{lm:sym}
If a centrally symmetric curve $\g$ possesses an invariant curve $\delta$ consisting of periodic points, then $\delta$ is centrally symmetric.
\end{lemma}

\proof
Let $S$ be the reflection in the origin. Since $\gamma$ is centrally-symmetric, 
$S$ and $T$ commute. Hence $S(\delta)$ is an invariant curve of $T$ with the same rotation number.

 We shall use the following consequence of the Aubry-Mather theory on the structure of minimal orbits with rational rotation number (see, for example, the survey paper \cite{bangert}, Theorem 5.8):

{\it If $\delta$ is a rotational invariant curve of a twist map of the cylinder with rational rotation number which consists solely of periodic orbits, then any other rotational invariant curve with the same rotation number must coincide with $\delta$.
}

{Indeed, it follows from this theory that all orbits lying on a rotational invariant curve are minimizing. Moreover, if the curve has a rational rotation number $\rho$ and all the orbits on this curve are periodic, then the set of these orbits coincides with the set 
of the minimizers  with the rotation number $\rho$}.

Applying this fact to the invariant curves $\delta$ and $S(\delta)$ we conclude that 
$\delta=S(\delta),$
as claimed.
\proofend

By the above lemma, the periodic quadrilateral must be parallelograms. 
Let us specialize the discussion of the preceding section to this particular situation. 

The space ${\C}$ of origin-centered parallelograms of a fixed perimeter is 3-dimensional. Such a parallelogram is characterized by four parameters $(\alpha_1,\alpha_2,p_1,p_2)$, see Figure \ref{para}. The radii of the circles are equal to $p_1$ and $p_2$, and the perimeter of the parallelogram equals $\frac{4(p_1+p_2)}{\sin(\alpha_2-\alpha_1)}$. Assume that the perimeter is 4, that is, $p_1+p_2 = \sin(\alpha_2-\alpha_1)$.

 \begin{figure}[ht]
\centering
\includegraphics[width=3.7in]{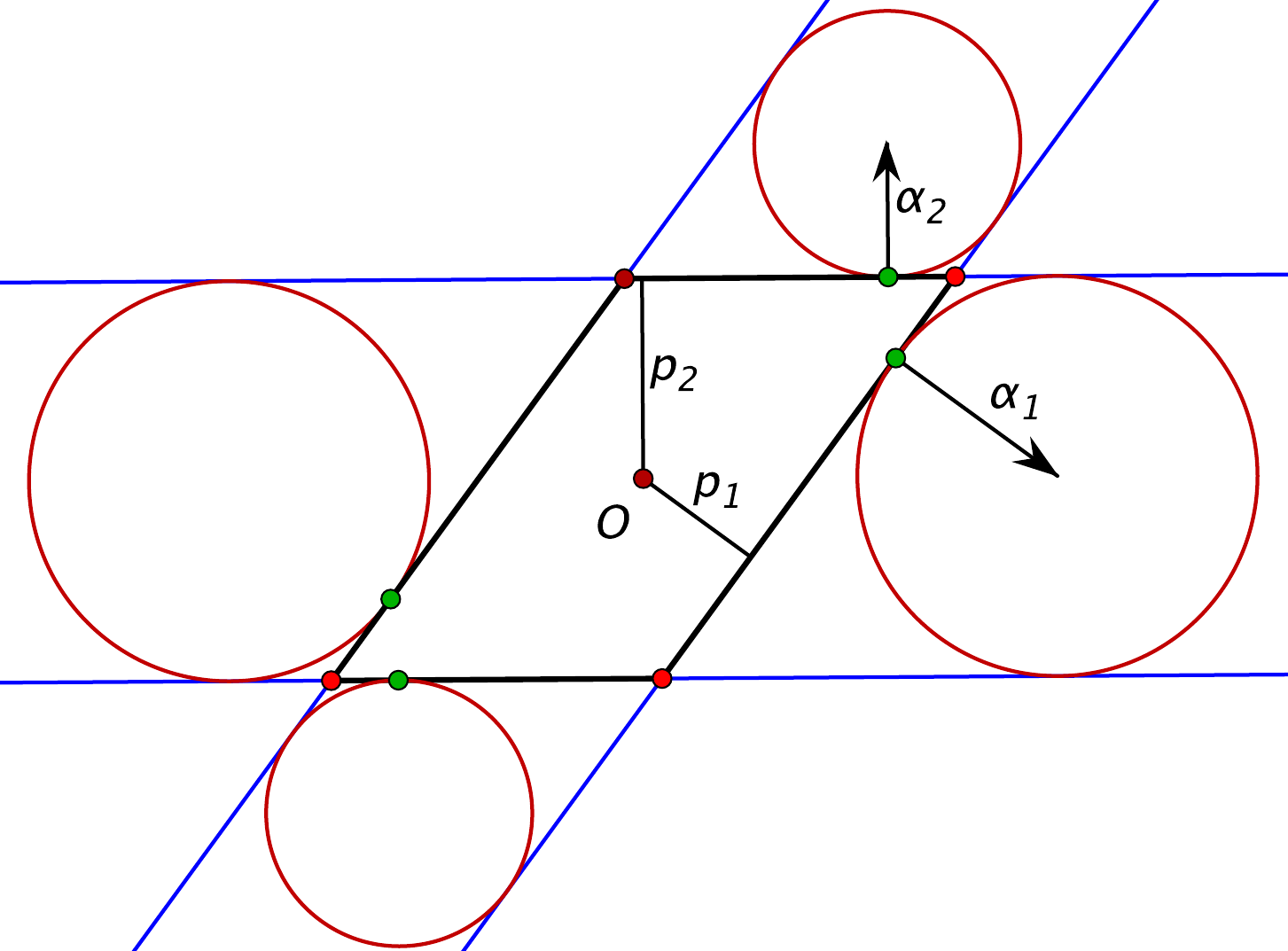}
\caption{Parameters of an origin centered parallelogram.}
\label{para}
\end{figure}

We apply Lemma \ref{lm:coord} in this situation. The angles involved are $\alpha_1,\alpha_2,\alpha_1+\pi,\alpha_2+\pi$, and the respective support numbers are $p_1,p_2,p_1,p_2$. A calculation yields

\begin{lemma} \label{lm:fields}
The 2-dimensional distribution $\D$ on $\C$ is spanned by the vector fields
\begin{equation*}
\begin{aligned}
&\xi_1 =\partial \alpha_1 - \cot(\alpha_2-\alpha_1) (p_1+p_2) \partial p_1 = \partial \alpha_1 - \cos(\alpha_2-\alpha_1) \partial p_1,\\
&\xi_2 =\partial \alpha_2 + \cot(\alpha_2-\alpha_1) (p_1+p_2) \partial p_2 = \partial \alpha_2 + \cos(\alpha_2-\alpha_1) \partial p_2,
\end{aligned}
\end{equation*}
and one has $[\xi_1,\xi_2]= \sin(\alpha_2-\alpha_1) (\partial p_1 - \partial p_2) \neq 0$. 
\end{lemma} 

We see that $\D$ is a completely non-integrable, i.e., it is a contact structure. As a contact form, one may take
$$
\lambda= \cos(\alpha_2-\alpha_1)  (d\alpha_1 + d\alpha_2) + dp_1-dp_2,
$$
whereas the differential of the perimeter, a  1-form that vanishes on $\C$, is 
$$
dF= \cos(\alpha_2-\alpha_1)  (d\alpha_1 - d\alpha_2) - dp_1-dp_2. 
$$

\begin{corollary} \label{cor:Iv}
There are no open sets consisting of 4-periodic orbits.
\end{corollary}

\proof
Such a set would be a surface tangent to the distribution $\D$; this would contradict its complete non-integrability.
\proofend

Now we shall parameterize the centrally symmetric smooth strictly convex outer length billiard curves $\g$ that possess invariant curves consisting of 4-periodic points. Let us change coordinates:
$$
x=\frac{\alpha_1+\alpha_2}{2},\ y=2\cos(\alpha_2-\alpha_1),\ z=p_2-p_1.
$$

Then $x\in [0,2\pi)$ is an angle, while $y$ and $z$ are reals. In these coordinates, the contact 1-form has the standard form $\lambda=ydx-dz$. We are interested in horizontal, Legendrian, curves, parameterized by the variable $x$; they are given by
$$
z=f(x), y = \frac{df(x)}{dx},
$$
where $f(x)$ is a differentiable function satisfying $|f'(x)|\le 2$ for all $x$.

In addition, we have an action of the cyclic group $\Z_4$ given by the action of its generator:
$$
(\alpha_1, \alpha_2, p_1, p_2) \mapsto (\alpha_2,\alpha_1+\pi, p_2, p_1),
$$
see Figure \ref{para}. The Legendrian curves involved are invariant under this action, hence, in the new coordinates, one has $f(x+\frac{\pi}{2}) = - f(x)$. Without loss of generality, we may assume that $f(0)=0$.

Expressing the old coordinates in terms of the new ones yields
\begin{equation*} \label{eq:old}
\begin{aligned}
&\alpha_1=x-\frac12 \arccos\left(\frac{f'(x)}{2}\right),\ \alpha_2=x+\frac12 \arccos\left(\frac{f'(x)}{2}\right),\\
&p_1=\frac{\sqrt{4-f'(x)^2}-f(x)}{4},\ p_2=\frac{\sqrt{4-f'(x)^2}+f(x)}{4}.
\end{aligned}
\end{equation*}

We arrive at the next result.

\begin{theorem}  \label{thm:4per}
Let $f(x)$ be a $2\pi$-periodic function satisfying $f(x+\frac{\pi}{2}) = - f(x)$ and $|f'(x)| < 2$ for all $x$. The centrally symmetric outer length billiard curve given by the parametric equation
$$
\g(x)= p (x) (\cos\alpha(x),\sin\alpha(x)) + \frac{p' (x)}{\alpha'(x)} (-\sin\alpha(x), \cos\alpha(x)),
$$
where 
$$
p(x)=\frac{\sqrt{4-f'(x)^2}-f(x)}{4},  \alpha(x)=x-\frac12 \arccos\left(\frac{f'(x)}{2}\right),
$$
possesses an invariant curve consisting of 4-periodic points, and all such outer length billiard curves are given by the above formula with a suitable function $f(x)$.
\end{theorem}

For example, if $\g$ is a circle, then $\alpha_2-\alpha_1=\frac{\pi}{2}$, and $z=y=0$, that is, $f(x)=0$ identically. Let $f(x)$ be a differentiable function satisfying $f(x+\frac{\pi}{2}) = - f(x)$ (i.e., its Fourier expansion contains only the 
$(2+4k)$th harmonics with $k\in \Z$; for example, $\sin 2x$ will do). Then, for sufficiently small $\eps$, the function $\eps f(x)$ will yield a desired oval. Thus one has a functional space of examples.

\begin{example} \label{ex:ell}
{\rm
Let us consider the case when $\g$ is an ellipse. The outer length billiard about an ellipse is completely integrable, see \cite{BBF} or \cite{ACT}, and the invariant curves are the confocal ellipses.  Therefore the 4-periodic trajectories of the outer length billiard about $\g$ are the 4-periodic billiard trajectories in a confocal ellipse  for which $\g$ serves as a caustic.

We use the information about 4-periodic orbits from \cite{BM}, and we use the notation therein. The support function of the origin-centered ellipse with the semi-axes $a>b$ is given by  $h(\psi)=\sqrt{a^2 \cos^2 \psi + b^2 \sin^2 \psi}$. Consider a 4-periodic orbit in this ellipse, and let two consecutive oriented lines that form this orbit have support coordinates $(\varphi,p)$ and $(\varphi_1,p_1)$. Set $\psi=\frac{\varphi+\varphi_1}{2}, d(\psi)=\frac{\varphi-\varphi_1}{2}$. Then 
$$
d(\psi)=\frac12 \arccos\left(\frac{b^2-a^2}{b^2+a^2}\cos 2\psi   \right).
$$

Let us relate these formulas to our notations above. We need the perimeter of the 4-periodic quadrilateral to be equal to four, hence $a^2+b^2=1$, or $a=\sin t, b= \cos t$ for some parameter $t$, see Figure \ref{per4}.
\begin{figure}[ht]
\centering
\includegraphics[width=1.8in]{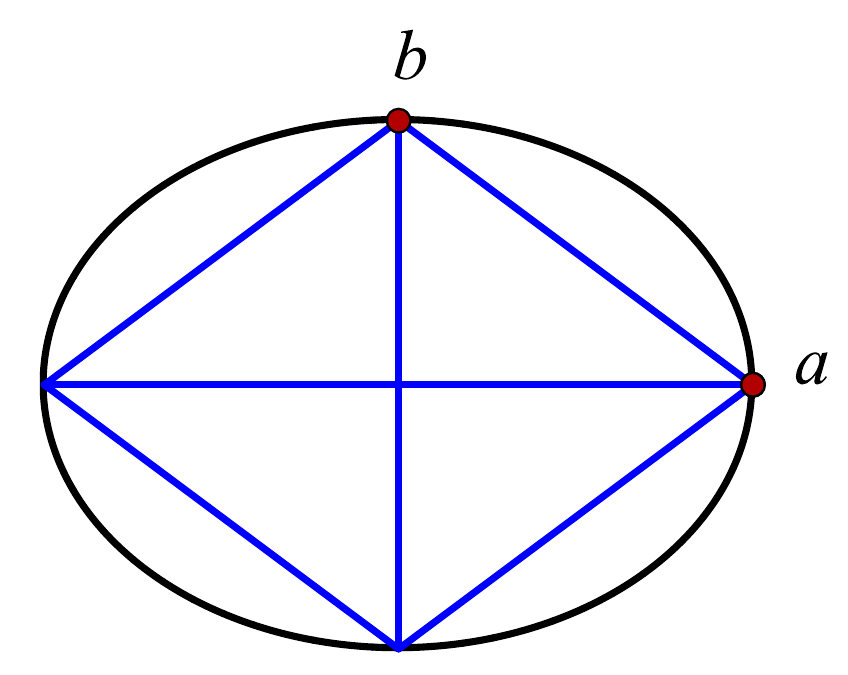}
\caption{A 4-periodic billiard orbit in an ellipse.}
\label{per4}
\end{figure}

We have 
$$
\psi=x,\ \varphi=\alpha_1,\ \varphi_1=\alpha_2,\ d=\frac{\alpha_2-\alpha_1}{2} = \frac12 \arccos\left(\frac{f'}{2}\right),
$$
and we conclude that $f(x)= \cos 2t \sin 2x$.
}
\end{example} 

Now we describe a construction analogous to the one used for Radon curves, cf. \cite{BBT}.

Since $R_0=p_1, R_2=p_0$ (see Figure \ref{para}) we get, using the formulas of Lemma \ref{lemma:partialsS},
$$
p_0'=-(p_0+p_1)\cot\omega, \ 
p_1'=(p_0+p_1)\cot\omega.
$$
Moreover,  since the perimeter of the parallelogram can be normalized to be equal to $4$, we have the following equations:
$$
p(\alpha_1)+p(\alpha_2)=\sin\omega,\ p'(\alpha_1)=-\cos\omega,\ p'(\alpha_2)=\cos\omega, \ \omega=\alpha_2-\alpha_1.
$$
Rotating  the plane, we may assume that $p'(0)=0$, thus $\alpha_2=\pi/2$ corresponds to $\alpha_1=0$,  and $p'(\pi/2)=0$ as  well.
We also need to choose $p(0), p(\pi/2)$ so that $p(0)+p(\pi/2)=1$.

Now take  any convex arc in the first quadrant defined by support function $p(\alpha), \alpha\in[0,\pi/2]$ satisfying the constraints
$$
p'(0)=p'(\pi/2)=0,\ p(0)+p(\pi/2)=1,
$$
and extend to the second quadrant as follows.
Set $\beta$ corresponding to $\alpha$: $$\beta\in [\pi/2, \pi],\quad \beta:=\alpha+\arccos(-p'(\alpha))$$ and $$
p(\beta):=-p(\alpha)+\sin(\beta-\alpha).
$$
Then we extend the curve from the upper half plane by central symmetry.

For example, we can start with a small perturbation of a circle or an ellipse and obtain, as a result, a convex closed curve.

\section{Two other proofs of the 3-periodic version of  Ivrii's conjecture} \label{sect:Ivr}

Following the maxim that it's better to have different proofs of the same theorem than the same proof of different theorems, we present two other proofs of the 3-periodic version of  Ivrii's conjecture for outer length billiard. The first is purely geometric; it is inspired by and similar to that in \cite{Sh}. The second proof is analogous to the ones in \cite{BZ,La,GT,TZ,Tu}.

\begin{theorem} \label{thm:Iv}
The set of 3-periodic points of an outer length billiard has empty interior. 
\end{theorem} 

\noindent{\bf First proof.}
Let $\g$ be an oval that possesses an open set of 3-periodic orbits. We work toward a contradiction.

Let $ABC$ be such orbit and let $U, V$ and $W$ be the tangency points of the excircles with the respective sides, see Figure \ref{def}. These points are also the tangency points of the sides of the triangle with $\g$. 

 \begin{figure}[ht]
\centering
\includegraphics[width=3.3in]{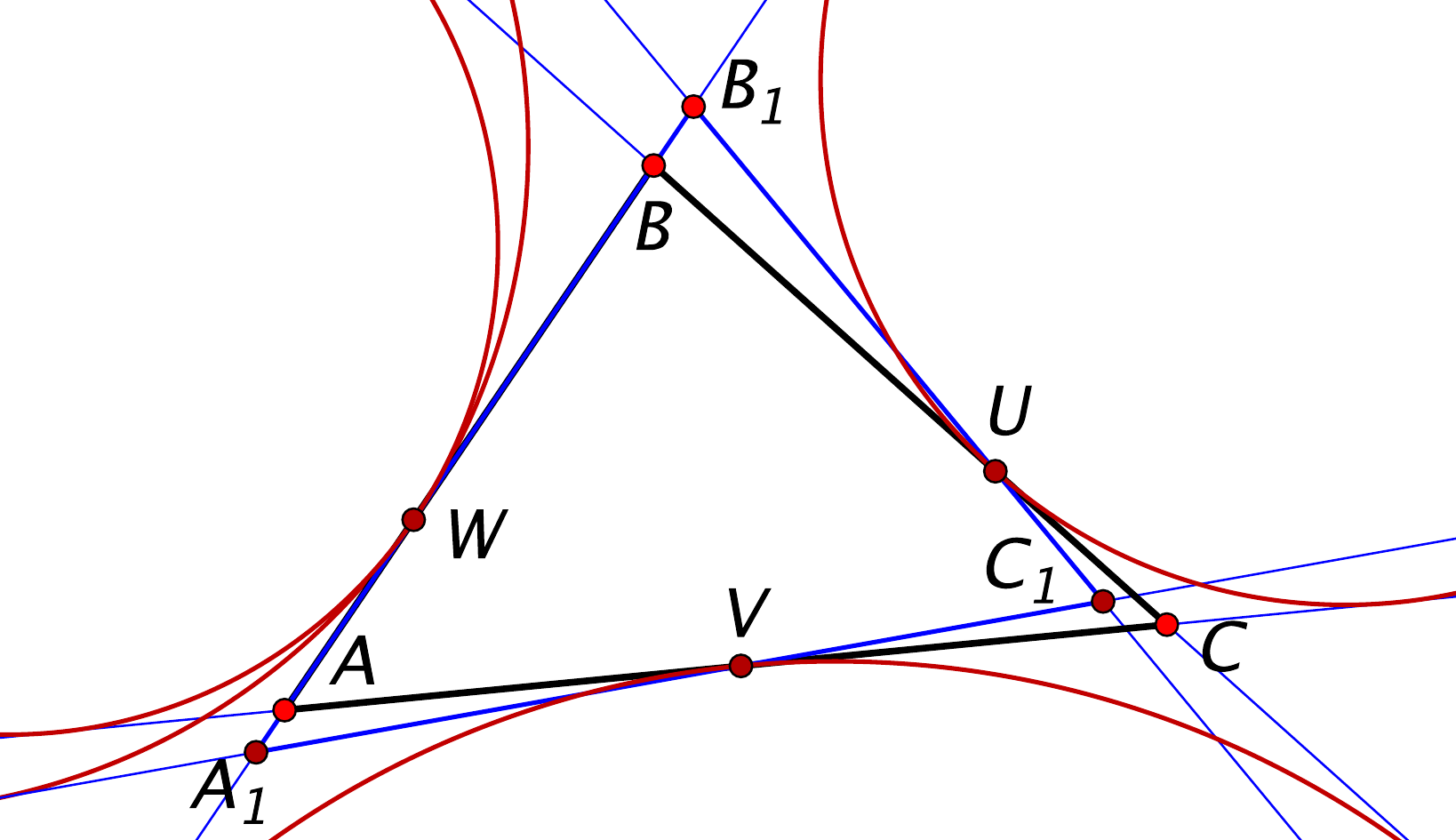}
\caption{A deformation of a 3-periodic orbit.}
\label{def}
\end{figure}

There exists an infinitesimal deformation of this orbit, $A_1B_1C_1$, where points $A_1$ and $B_1$ lie on the side $AB$. The corresponding infinitesimal deformation of the sides $AC$ and $BC$ are their rotations about points $V$ and $U$.

First, we claim that if point $A_1$ has moved from point $W$, then point $B_1$ has also moved from point $W$, as shown in Figure \ref{def}. Therefore the sides $AC$ and $BC$ rotate counterclockwise and clockwise, respectively. 

\begin{figure}[ht]
\centering
\includegraphics[width=3.7in]{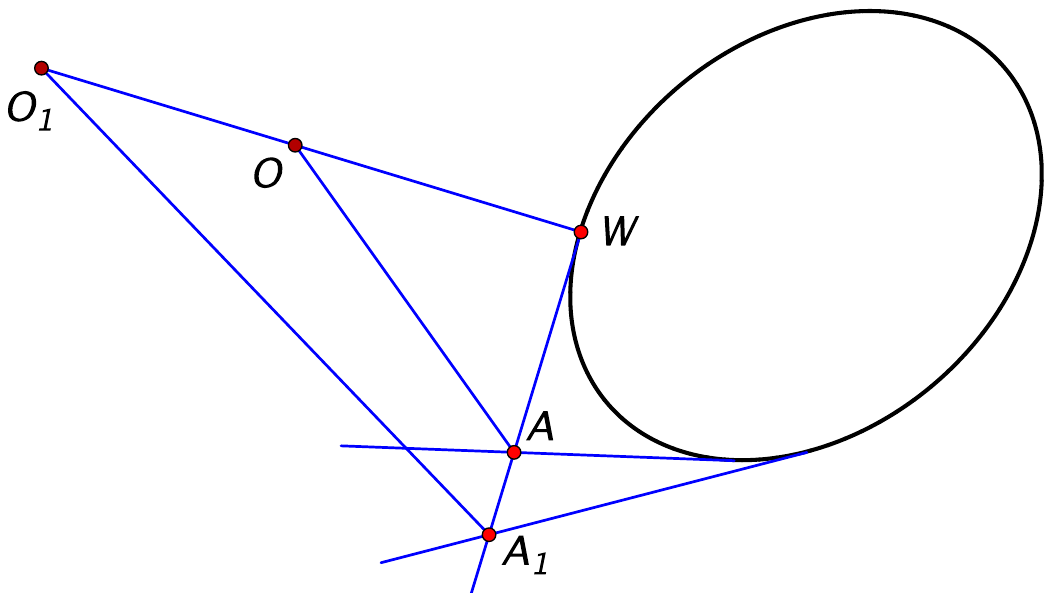}
\caption{Moving point $A$.}
\label{move}
\end{figure}

Indeed, the point $W$ is the point of tangency of the tangent line to $\g$ from point $A$, hence it stays put. If  point $A$ moves from point $W$, then the angle under which $\g$ is seen from this point becomes smaller, and the exterior angle gets bigger. Therefore the bisector of the exterior angle  is turning in the positive sense, and its intersection point with the normal to $\g$ at point $W$ moves farther from $\g$ -- these intersection points are the centers of the respective circles, points $O$ and $O_1$ in Figure \ref{move}. 
Conversely, of point $B$ moves toward point $W$, the center $O$ of the circle also moves toward point $W$. As we saw, this is not the case,  hence point $B$  moves from point $W$, as claimed. 

Next let us examine the infinitesimal motions of point $U$ in two cases: when the line $AC$ rotates counterclockwise about point $V$, while the other sides are fixed, and when the line $BC$ rotates clockwise about point $U$, while the other sides are fixed. We claim that, in both cases, the velocity of point $U$ has the direction toward point $B$.

 \begin{figure}[ht]
\centering
\includegraphics[width=3in]{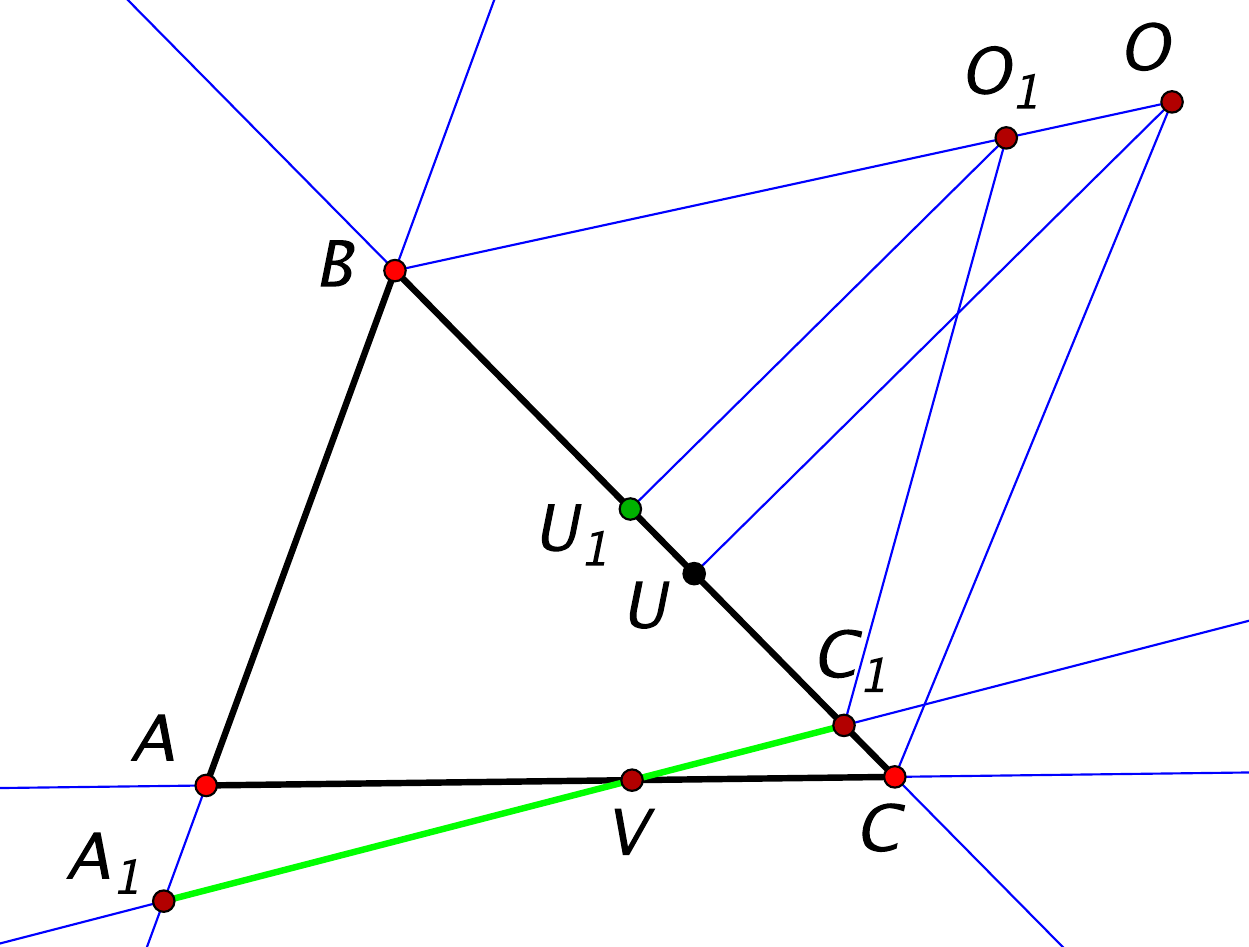}
\caption{Rotation of the line $AC$.}
\label{AC}
\end{figure}

The first case is depicted in Figure \ref{AC}.  The bisector of the exterior angle, $CO$, becomes $C_1O_1$, point $C$ is moving toward point $B$ and the bisector is rotating counterclockwise. Therefore point $O$, the center of the circle tangent the line $BC$ at point $U$, moves toward point $B$, to point $O_1$, and  hence point $U$, its projection on line $BC$,  moves toward point $B$, to $U_1$,  as claimed. 

\begin{figure}[ht]
\centering
\includegraphics[width=3.3in]{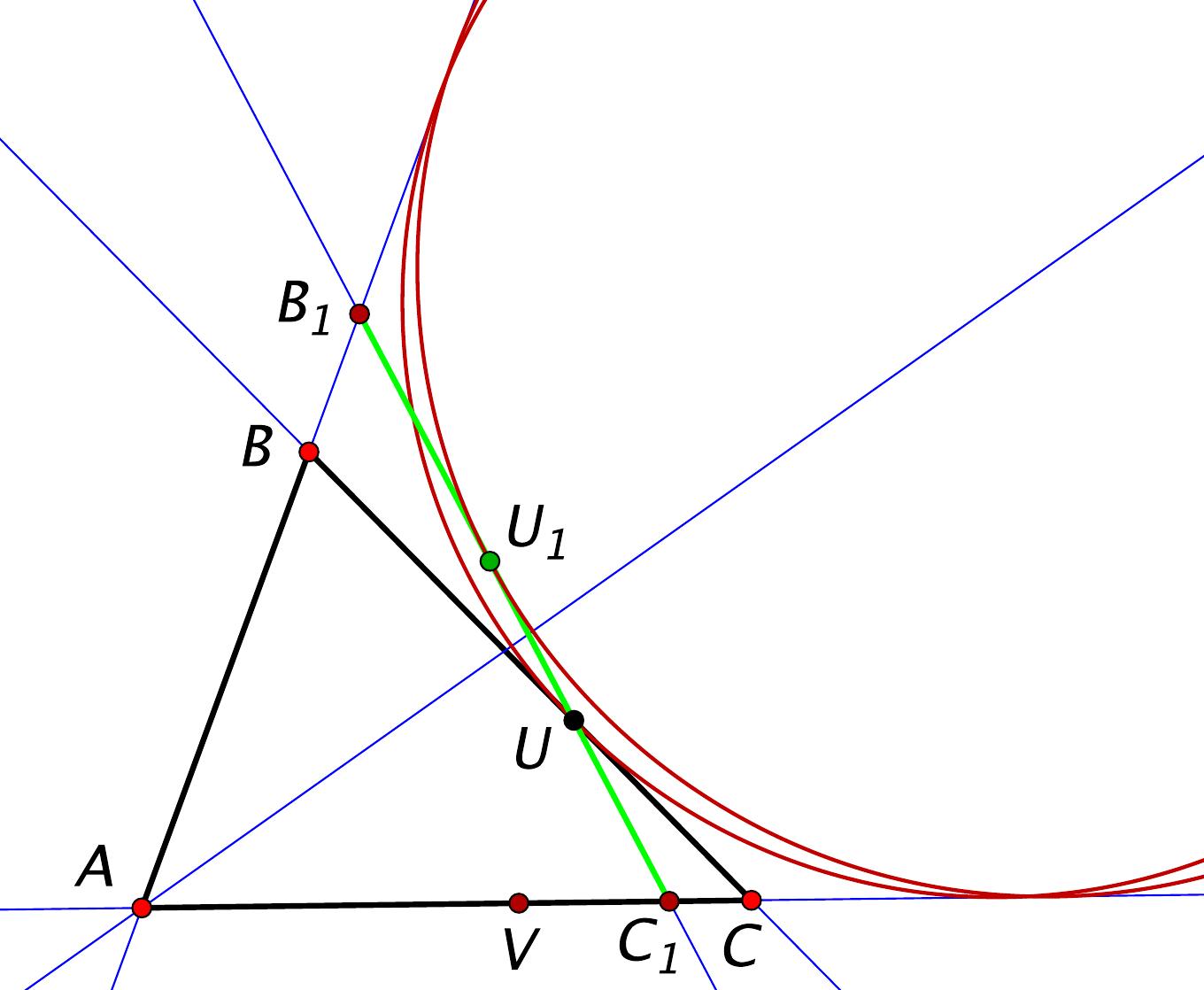}
\caption{Rotation of the line $BC$.}
\label{BC}
\end{figure}

The second case is depicted in Figure \ref{BC}. Recall an elementary geometry problem: given a point $U$ inside the angle with vertex $A$, construct a chord $BC$ through point $U$ that minimizes the perimeter of $\triangle BAC$. 
The solution is to inscribe a circle in the angle so that point $U$ lies between the circle and the vertex $A$, and then to push this circle in until it passes through point $U$. The tangent line $BC$ to the circle at $U$ is the solution. We encountered this argument earlier, in the proof of Lemma \ref{lm:per}.

Consider the infinitesimally deformed  $\triangle B_1A C_1$. The point $U$ is now outside its excircle. 
If the line $BC$ has turned clockwise about point $U$, then the tangency point of the excircle $U_1$  lies between $U$ and $B_1$ (and if the rotation is counterclockwise, the tangency point lies between $U$ and $C_1$). Hence the velocity of point $U$ points toward point $B$, as claimed. 

Finally, to a contradiction. When $\triangle ABC$ is infinitesimally deformed to $\triangle A_1 B_1 C$ (see Figure \ref{def}), the line $AC$ is rotating counterclockwise and the line $BC$ clockwise. The velocities of point $U$ under these two motions add up and, as we saw, the resulting motion of the tangency  point $U$ is toward point $B$. However consider Figure \ref{rot}. The convexity of the curve $\g$ implies that the vector $U U_1$ has the direction toward point $C$, that is, from point $B$. This is the desired contradiction. 
\proofend

\begin{figure}[ht]
\centering
\includegraphics[width=4in]{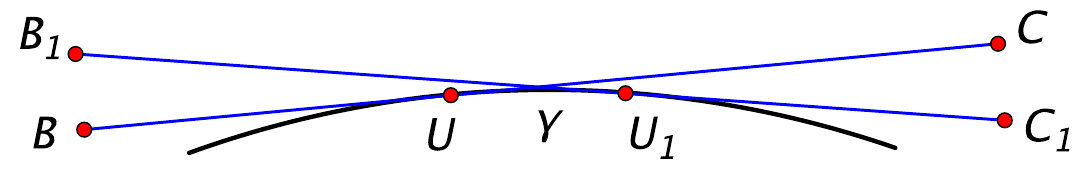}
\caption{The convexity of the curve $\g$ implies that  point $U$ moves in the direction from point $B$.}
\label{rot}
\end{figure}

\noindent {\bf Second proof.}  We specialize the formulas  of Section \ref{sect:dis} in the case $n=3$. Let $\C$ be the space of triangles of unit perimeter, a 5-dimensional space, and $\D$ be the 3-dimensional distribution therein,  generated by the vector fields $\xi_i=\partial \alpha_i + \Phi_i \partial p_i,\ i=1,2,3$, where $\Phi_i$ is the right hand side expression in (\ref{eq:cr}).

One has
$$
[\xi_1,\xi_2]=\left(\frac{\partial \Phi_2}{\partial \alpha_1} + \Phi_1 \frac{\partial \Phi_2}{\partial p_1}  \right) \partial p_2
- \left(\frac{\partial \Phi_1}{\partial \alpha_2} + \Phi_2 \frac{\partial \Phi_1}{\partial p_2}  \right) \partial p_1,
$$
and two other similar formulas obtained by the cyclic permutations of the indices. 
To simplify the notations, set
$$
W_2=\frac{\partial \Phi_2}{\partial \alpha_1} + \Phi_1 \frac{\partial \Phi_2}{\partial p_1},\ U_1= \frac{\partial \Phi_1}{\partial \alpha_2} + \Phi_2 \frac{\partial \Phi_1}{\partial p_2},
$$ 
and four similar formulas obtained by the cyclic permutations of the indices. We shall calculate these quantities. 

We are free to choose the origin (this affects the values of $p_i$), and a convenient choice is the incenter, whence $p_i=r, i=1,2,3$. By scaling, we may assume that $r=1$. We note that the operations of taking partial derivative with respect to $\alpha_i$ and setting $p_j=1$ for all $j$ commute, and this simplifies  calculations.

To further simplify the notations, denote the exterior angles of the triangle by $2u,2v,2w$, so that
$$
\alpha_2-\alpha_1=2u,\ \alpha_3-\alpha_2=2v,\ \alpha_1-\alpha_3=2w.
$$
Note that $u+v+w=\pi$ and that
\begin{equation} \label{eq:part}
\partial \alpha_1 = \frac12 (\partial w - \partial u),\ \partial \alpha_2 = \frac12 (\partial u - \partial v),\
\partial \alpha_2 = \frac12 (\partial v - \partial w). 
\end{equation}

One has
$$
\Phi_2=\frac{\cos^2 u\ (p_{3}+p_2) - \cos^2 v\ (p_{2}+p_1) }
{2\sin (u+v) \cos u \cos v},
$$
hence
$$
\frac{\partial \Phi_2}{\partial p_1} = -\frac{\cos v}{2\sin (u+v) \cos u}, 
\frac{\partial \Phi_2}{\partial p_2} = \frac{\cos^2u -\cos^2 v}{2\sin (u+v) \cos u\cos v},
\frac{\partial \Phi_2}{\partial p_3} = \frac{\cos u}{2\sin (u+v) \cos v},
$$
and, setting  $p_j=1$ for all $j$,
$$
\frac{\partial \Phi_2}{\partial \alpha_1}=\frac{1}{2\cos^2 u}, 
\frac{\partial \Phi_2}{\partial \alpha_2}=-\frac12 \left(\frac{1}{2\cos^2 u}+\frac{1}{2\cos^2 v}\right),
\frac{\partial \Phi_2}{\partial \alpha_3}=\frac{1}{2\cos^2 v}.
$$
As before, other similar formulas are obtained by the cyclic permutations of the indices.

Then, using some trigonometry and remembering that $u+v+w=\pi$, we obtain
\begin{equation} \label{eq:WU}
\begin{aligned}
&W_1= \frac{\tan u}{\sin 2v}, W_2=\frac{\tan v}{\sin 2w}, W_3=\frac{\tan w}{\sin 2u},\\
&U_1=\frac{\tan w}{\sin 2v},   U_2=\frac{\tan u}{\sin 2w}, U_3=\frac{\tan v}{\sin 2u},
\end{aligned}
\end{equation}
all positive quantitives.

We note, in passing, that these formulas imply that $\D$ has the growth type $(3,5)$. Indeed,  $\D^\perp$ is generated by any two of the three differential 1-forms ${r_1}_i = dp_i - \Phi_i d\alpha_i$. We see that 
$$
\theta_1 ([\xi_1,\xi_2]) < 0, \theta_2 ([\xi_1,\xi_2]) > 0, \theta_1 ([\xi_2,\xi_3]) = 0, \theta_2 ([\xi_2,\xi_3]) < 0.
$$
It follows that the fields $[\xi_1,\xi_2]$ and  $[\xi_2,\xi_3]$, together with $\D$, span the tangent spaces of $\C$.

We continue the proof of the 3-periodic Ivrii conjecture for outer  length billiards. 
Assume that there exists an open disc consisting of 3-periodic points. Then one has a horizontal disc  $D\subset \C$ spanned by a linear combination of the vector fields $\xi_i$, say $
a_1 \xi_1+a_2\xi_2+a_3\xi_3\ \ {\rm and}\ \ b_1 \xi_1+b_2\xi_2+b_3\xi_3.
$
Without loss of generality,  assume that $a_1 \ne 0, b_2 \ne 0$ and, taking linear combinations,  replace these fields by $\xi_1 + a \xi_3$ and $\xi_2+b\xi_3$. Then 
$$
[\xi_1 + a \xi_3, \xi_2 + b \xi_3] = f (\xi_1 + a \xi_3) + g (\xi_2 + b \xi_3).
$$
The bracket on the left equals
$$
[\xi_1,\xi_2]-a[\xi_2,\xi_3]-b[\xi_3,\xi_1] + \{\xi_1(b)+a\xi_3(b)-\xi_2(a)-b\xi_3(a)\}\, \xi_3,
$$
and this expression does not contain the terms with $\partial \alpha_1$ and $\partial \alpha_2$. Hence $f=g=0$, and we
conclude that
\begin{equation} \label{eq:both}
[\xi_1,\xi_2]-a[\xi_2,\xi_3]-b[\xi_3,\xi_1] =0,\ \xi_1(b)+a\xi_3(b)-\xi_2(a)-b\xi_3(a) = 0.
\end{equation}
The first equation reads
$$
W_2 \partial p_2 - U_2 \partial p_1 - a(W_3 \partial p_3 - U_2 \partial p_2) - b (W_1 \partial p_1 - U_3 \partial p_3) =0,
$$
and solving it yields
$$
a = - \frac{W_2}{U_2} = - \frac{\tan v}{\tan u},\ b = - \frac{U_1}{W_1} = - \frac{\tan w}{\tan u}.
$$

Recall the equalities (\ref{eq:part}) and that $\xi_i=\partial \alpha_i + \Phi_i \partial p_i$.
We are in a position to compute (twice) the left hand side of the second equality in (\ref{eq:both}):
\begin{equation*}
\begin{aligned}
&-(\partial w - \partial u) \left(\frac{\tan w}{\tan u}\right) + \frac{\tan v}{\tan u} (\partial v - \partial w) \left(\frac{\tan w}{\tan u}\right)\\
&+ (\partial u - \partial v) \left(\frac{\tan v}{\tan u}\right)
 - \frac{\tan w}{\tan u} (\partial v - \partial w) \left(\frac{\tan v}{\tan u}\right) =\\ 
 &-\frac{1}{\cos^2 w \tan u} -\frac{\tan w}{\sin^2 u} -\frac{\tan v}{\cos^2 w \tan^2 u} 
 -\frac{\tan v}{\sin^2 u} -\frac{1}{\cos^2 v \tan u} -\frac{\tan w}{ \cos^2 v\tan^2 u}.
\end{aligned}
\end{equation*} 
Every term of this expression is negative, so it is not equal to zero. This contradiction completes the proof.
\proofend

\begin{figure}[ht]
\centering
\includegraphics[width=2.3in]{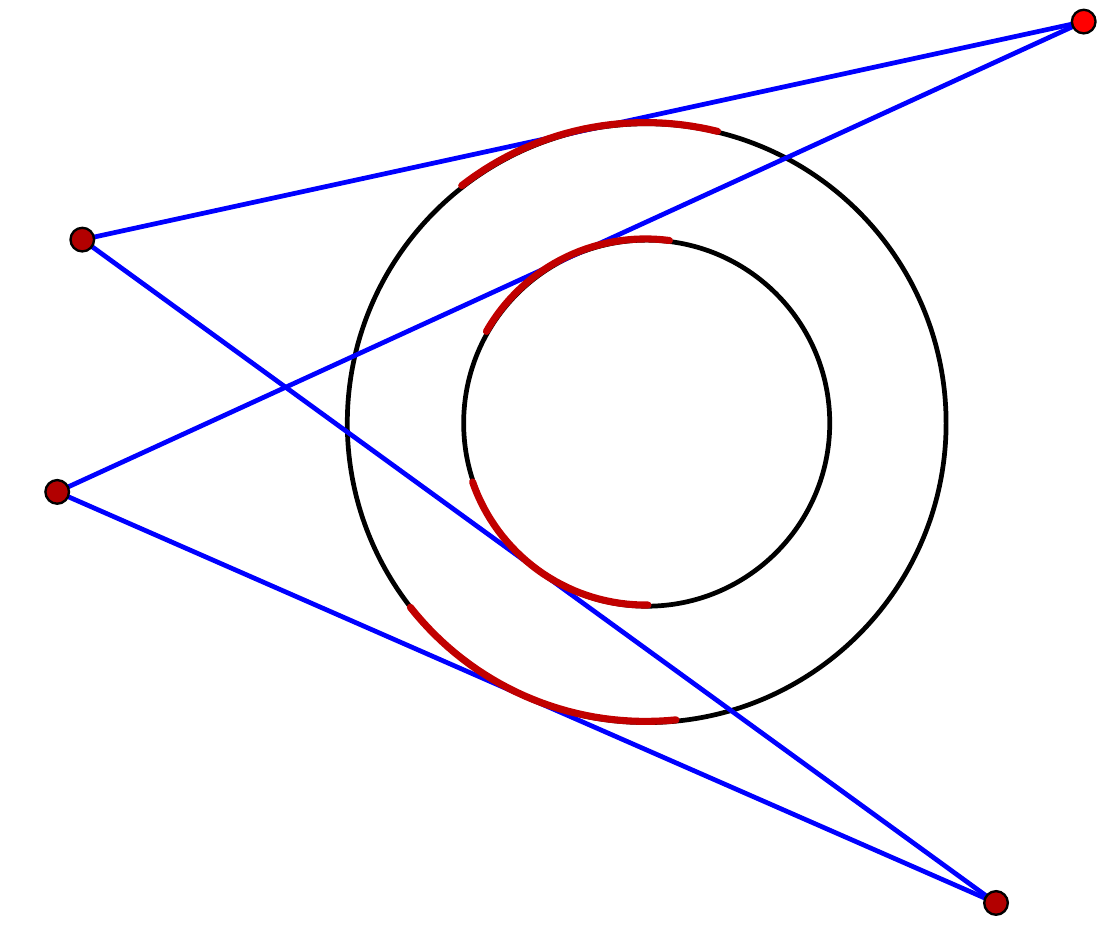}
\caption{The composition of the outer length billiard reflections in the concentric circular arcs is the identity.}
\label{circ}
\end{figure}

\begin{remark} 
{\rm 
1. A completely non-integrable 3-dimensional distribution in 5-dimensional manifold may admit a horizontal surface. 
Consider 5-dimensional space with coordinates $x,y,z,u,v$ and a 3-dimensional distribution $\D$ spanned by the vector fields $\partial x, \partial y, \partial z + x \partial u + y \partial v$. Then $\D$ is non-integrable with  the bracket growth type $(3,5)$, but $\D$ admits a horizontal 2-dimensional disc spanned by the commuting vector fields  $\partial x$ and  $\partial y$.

2. One can construct four germs of properly oriented convex curves so that the composition of the respective four outer length billiard reflections  is the identity map. It suffices to take a pair of arcs on each of two concentric circles. The respective outer length billiard maps are rotations about the common center of the circles, and such rotations commute, see 
Figure \ref{circ}. A similar consideration applies to other billiard models, including Birkhoff billiards (where one could also use confocal ellipses instead of concentric circles).
}
\end{remark}

\end{document}